\documentclass[reqno]{amsart}

\usepackage[utf8]{inputenc}
\usepackage{amssymb, amsmath, amsthm, color, enumerate, mathabx, mathrsfs}
\usepackage{bbm}
\usepackage{pdfpages}
\usepackage{tcolorbox}
\usepackage{esint} 

\numberwithin{equation}{section}
\theoremstyle{plain}
\newtheorem{theorem}{Theorem}[section]
\newtheorem{lemma}[theorem]{Lemma}
\newtheorem{proposition}[theorem]{Proposition}

\newtheorem{conjecture}[theorem]{Conjecture}

\theoremstyle{definition}
\newtheorem{definition}[theorem]{Definition}

\newcommand{\bbF}{\mathbb{F}}

\newcommand{\N}{\mathbb{N}}

\newcommand{\R}{\mathbb{R}}

\newcommand{\Z}{\mathbb{Z}}

\newcommand{\cE}{\mathcal{E}}

\newcommand{\cH}{\mathcal{H}}

\newcommand{\cM}{\mathcal{M}}

\newcommand{\cP}{\mathcal{P}}

\newcommand{\cS}{\mathcal{S}}

\newcommand{\sE}{\mathscr{E}}

\newcommand{\bdr}{\mathbf{r}}

\newcommand{\bdone}{\mathbf{1}}

\newcommand{\bddelta}{\boldsymbol{\delta}}

\newcommand{\fkR}{\mathfrak{R}}

\newcommand{\ud}{\mathrm{d}}
\def\inn#1#2{\langle#1,#2\rangle}

\newcommand{\supp}{\mathrm{supp}\,}
\newcommand{\diam}{\mathrm{diam}\,}

\newcommand{\tang}{\mathrm{tang}}
\newcommand{\trans}{\mathrm{trans}}

\begin{document}

\title[Bounds for the strong spherical maximal operator]{Improved $L^p$ bounds for the strong spherical maximal operator}

\date{}

\begin{abstract} We study the $L^p$ mapping properties of the strong spherical maximal function, which is a multiparameter generalisation of Stein's spherical maximal function. We show that this operator is bounded on $L^p$ for $p > 2$ in all dimensions $n \geq 3$. This matches the conjectured sharp range $p>(n+1)/(n-1)$ when $n=3$. For $n=2$ the analogous estimate was recently proved by Chen, Guo and Yang.

Our result builds upon and improves an earlier bound of Lee, Lee and Oh. The main novelty is an estimate in discretised incidence geometry that bounds the volume of the intersection of thin neighbourhoods of axis-parallel ellipsoids. This estimate is then interpolated with the Fourier analytic $L^p$-Sobolev estimates of Lee, Lee and Oh. 
\end{abstract}

\author[J. Hickman]{ Jonathan Hickman }
\address{School of Mathematics and Maxwell Institute for Mathematical Sciences, James Clerk Maxwell Building, The King's Buildings, Peter Guthrie Tait Road, Edinburgh, EH9 3FD, UK.}
\email{jonathan.hickman@ed.ac.uk}

\author[J. Zahl]{Joshua Zahl}
\address{Department of Mathematics, University of British Columbia, Vancouver BC, V6T 1Z2, Canada}
\email{jzahl@math.ubc.ca}

\maketitle




\section{Introduction} 




\subsection{The strong spherical maximal operator} Let $n \geq 2$ and for a tuple of dilation parameters $\bdr = (r_1, \dots, r_n) \in (0,\infty)^n$, define the dilation map
\begin{equation*}
   \bddelta_{\bdr} \colon \R^n \to \R^n, \qquad \bddelta_{\bdr}(x_1, \dots, x_n) := (r_1x_1, \dots, r_nx_n). 
\end{equation*}
Let $\sigma$ denote normalised (to have mass 1) surface measure on the unit sphere $S^{n-1}$. The \textit{strong spherical maximal function} is then given by 
\begin{equation*}
    \cM_{\mathrm{st}}f(x) := \sup_{\bdr \in (0,\infty)^n} |f \ast \sigma_{\bdr}(x)| \qquad \textrm{for $f \in C_c(\R^n)$,}
\end{equation*}
where the dilated measure $\sigma_{\bdr}$ is defined via the action
\begin{equation*}
    \inn{\sigma_{\bdr}}{f} := \inn{\sigma}{f \circ \bddelta_{\bdr}} \qquad \textrm{for all $f \in C_c(\R^n)$.}
\end{equation*}
Thus, the operator $\cM_{\mathrm{st}}$ forms maximal averages over concentric axis-parallel ellipsoids and is a natural multiparameter variant of Stein's spherical maximal function~\cite{Stein1976}. Furthermore, $\cM_{\mathrm{st}}$ can be interpreted as a singular analogue of the classical strong maximal function (the latter is the analogous multiparameter variant of the Hardy--Littlewood maximal function).\medskip 




\subsection{The main result} This paper concerns the $L^p$ mapping properties of $\cM_{\mathrm{st}}$. Our main result reads as follows.

\begin{theorem}\label{thm: main} For all $n \geq 3$ and $p > 2$, there exists a constant $C_{n,p} > 0$ such that
\begin{equation*}
    \|\cM_{\mathrm{st}}f\|_{L^p(\R^n)} \leq C_{n,p} \|f\|_{L^p(\R^n)} \qquad \textrm{holds for all $f \in C_c(\R^n)$.}
\end{equation*}  
\end{theorem}

Theorem~\ref{thm: main} improves upon an earlier result of Lee, Lee and Oh \cite{LLO}, which established $L^p$-boundedness for $p > 2 \cdot \frac{n+1}{n-1}$. Moreover, Theorem~\ref{thm: main} establishes the sharp range of estimates for $n = 3$. However, it is likely that the result is not sharp in higher dimensions, and it is natural to conjecture the following.

\begin{conjecture}[Strong spherical maximal function conjecture]\label{conj: main} For all $n \geq 2$ and $p > \frac{n+1}{n-1}$, there exists a constant $C_{n,p} > 0$ such that
\begin{equation}\label{eq: strong max conj}
    \|\cM_{\mathrm{st}}f\|_{L^p(\R^n)} \leq C_{n,p} \|f\|_{L^p(\R^n)} \qquad \textrm{holds for all $f \in C_c(\R^n)$.}
\end{equation}
\end{conjecture}

The conjectural range $p > \frac{n+1}{n-1}$ is necessary, in the sense that the inequality \eqref{eq: strong max conj} fails for all $p \leq \frac{n+1}{n-1}$. This can be shown using a modified Knapp-type example: we provide a heuristic for this example below in \S\ref{subsec: method} and a precise discussion in Appendix~\ref{app: necessity}. 

For $n = 2$, the conjecture is known from \cite{LLO_elliptic, CGY, PYZ}. More precisely, the $n=2$ case follows by interpolating the endpoint result from \cite{PYZ} with the Fourier analytic $L^p$-Sobolev estimates from either \cite{LLO_elliptic} or \cite{CGY}. Theorem~\ref{thm: main} verifies Conjecture~\ref{conj: main} in the $n=3$ and provides partial progress for $n \geq 4$. 

We remark that there is a well-developed and classical theory of such multiparameter maximal functions in the lacunary setting (where the radii $r_j$ are restricted to powers of $2$) due to Ricci and Stein \cite[Theorem 3.2]{RS1992}.




\subsection{A discretised estimate}\label{sec: discretised estimate} The main novelty in our approach, compared with the earlier work in \cite{LLO}, is to use a geometric argument to prove an $L^2$ estimate for a discretised variant of $\cM_{\mathrm{st}}$. To describe the setup, given $x = (x_1, \dots, x_n) \in \R^n$ and $\bdr = (r_1, \dots, r_n) \in (0, \infty)^n$, we let $E(x;\bdr)$ denote the axis-parallel ellipsoid with centre $x$ and principal radii $\bdr$. More precisely,
\begin{equation}\label{eq: defining fn}
    E(x;\bdr) := \big\{ y \in \R^n : F_{x,\bdr}(y) = 0 \big\} \quad \textrm{where} \quad F_{x,\bdr}(y) := \sum_{j=1}^n \frac{(y_j - x_j)^2}{r_j^2} -1.
\end{equation}
Furthermore, for $\delta > 0$, we consider the \textit{ellipsoidal annulus} (or \textit{homoeoid})
\begin{equation*}
E^{\delta}(x, \bdr) := \big\{   y \in \R^n : |F_{x,\bdr}(y)| < \delta \big\};
\end{equation*}
by the regularity of the defining function, when $\bdr\in[1,2]^n$ the set $E^{\delta}(x, \bdr)$ is comparable to the $\delta$-neighbourhood of $E(x, \bdr)$. In order to study $\cM_{\mathrm{st}}$, we consider the localised and discretised operators 
\begin{equation}\label{eq: defining M delta f}
    M^{\delta}f(x) := \sup_{\bdr \in [1, 2]^n} \fint_{E^{\delta}(x,\bdr)} |f|. 
\end{equation}
The main ingredient in the proof of Theorem~\ref{thm: main} is the following.

\begin{theorem}[Discretised $L^2$ bound]\label{thm: discretised} For all $n \geq 3$ and all $\varepsilon > 0$, there exists a constant $C_{n,\varepsilon} > 0$ such that
    \begin{equation}\label{eq: discretised global}
    \|M^{\delta}f\|_{L^2(\R^n)} \leq C_{n,\varepsilon} \delta^{-\varepsilon} \|f\|_{L^2(\R^n)}
\end{equation}
holds for all $0 < \delta < 1$ and all $f \in L^2(\R^n)$.
\end{theorem}

Owing to both the $\delta^{-\varepsilon}$-loss in the operator norm and the localisation of the radii, Theorem~\ref{thm: discretised} does not directly imply Theorem~\ref{thm: main}. However, the discretised $L^2$ bound \eqref{eq: discretised global} can be translated into an $L^p$--Sobolev bound for a localised variant of $\cM_{\mathrm{st}}$ that we will call $\cM_{\mathrm{loc}}$; a precise definition will be given in Section \ref{subsec: loc max op}. This in turn can be interpolated against $L^p$-Sobolev estimates from \cite{LLO}, leading to favourable $L^p$--Sobolev estimates for $\cM_{\mathrm{loc}}$ in the range $p > 2$. Finally, using multiparameter Littlewood--Paley theory, as in \cite{LLO}, we pass from $\cM_{\mathrm{loc}}$ to $\cM_{\mathrm{st}}$ and thereby arrive at Theorem~\ref{thm: main}. The details of these Fourier analytic arguments are presented in \S\ref{sec: Fourier analysis}. 




\subsection{Methodology}\label{subsec: method} Here we discuss the key ideas of the proof of Theorem~\ref{thm: discretised}, which involves a \textit{variable slicing} argument. Variants of this basic approach were recently used by the first author and Jan{\v c}ar to study Stein's spherical maximal function \cite{HJ}, by Chen, Guo and Yang to solve the $n=2$ case of Conjecture~\ref{conj: main}  \cite{CGY}, and by the second author to study planar maximal operators \cite{Zahl_plane}. That said, the details of the argument used here differ significantly from these works.\medskip

\noindent \textit{The Knapp example.} To motivate the argument, it is useful to first consider the basic Knapp example, which provides the necessary condition $p \geq \frac{n+1}{n-1}$ for \eqref{eq: strong max conj}. Here we give a heuristic for the Knapp example using parameter counting; we provide a more thorough presentation of the example in Appendix~\ref{app: necessity}, where we also rule out the endpoint case $p = \frac{n+1}{n-1}$. Fix a centre $x \in \R^n$ with, say, $3/2 < |x| < 5/2$ and a point $\omega \in S^{n-1}$ on the unit sphere. We consider all ellipsoids $E(x,\bdr)$ centred at $x$; this is an $n$-dimensional family, parametrised by the principal radii $\bdr = (r_1, \dots, r_n)$. If we impose the condition that $\omega \in E(x, \bdr)$, then this corresponds to a single constraint $F_{x, \bdr}(\omega) = 0$ on $\bdr$, where $F_{x, \bdr}$ is the defining function in \eqref{eq: defining fn}. Furthermore, if we impose the condition that $E(x,\bdr)$ is tangent to $S^{n-1}$ at $\omega$, then this corresponds to $n-1$ additional constraints. Thus, we have $n$ parameters and $n$ constraints. We therefore expect that, for many choices of $x$, there should be a choice of $\bdr$ such that $E(x, \bdr)$ is tangent to $S^{n-1}$ at the specific point $\omega$ (this is made precise in Lemma~\ref{lem: tangencies}).\medskip

For $0 < \delta < 1$, we let $K(\delta)$ be a rectangular slab of dimensions $\delta \times \delta^{1/2} \times \cdots \times \delta^{1/2}$ centred at $\omega$ and lying parallel to the tangent plane to $S^{n-1}$ at $\omega$. By formalising the above observations, we may find a set $F \subseteq \R^n$ of ($n$-dimensional) measure roughly equal to $1$ such that for each $x \in F$ there exists an ellipsoid centred at $x$ which intersects $K(\delta)$ on a set of (surface) measure roughly equal to $\delta^{(n-1)/2}$. Taking $f := \chi_{K(\delta)}$ to be the characteristic function of $K(\delta)$ and plugging this into \eqref{eq: strong max conj}, we arrive at the necessarily condition $p \geq \frac{n+1}{n-1}$.\medskip

\noindent \textit{Higher order tangency.} We would be able to strengthen the Knapp example, and thereby disprove Conjecture~\ref{conj: main}, if it were possible to ensure the ellipsoids $E(x, \bdr)$ above are not only tangent to $S^{n-1}$ at $\omega$, but are \textit{higher order tangent} (for instance, in addition one of the principal curvatures of the ellipsoid is equal to $1$ at $\omega$). However, imposing a higher order tangency condition leads to a situation where the number of constraints is larger than the number of free parameters, which should effectively rule out a counterexample of this kind. Our argument is based on making this simple idea precise.\medskip

\noindent \textit{Duality and exceptional sets.} To prove Theorem~\ref{thm: discretised}, we apply a standard duality argument, following Cordoba's \cite{Cordoba1977} classical proof of the $L^2$ boundedness of the Kakeya maximal function (see also \cite{Carbery1988} or \cite[Chapter 22]{Mattila_book}). This reduce matters to bounding expressions of the form 
\begin{equation}\label{eq: multiplicity}
    \big\| \sum_{E \in \sE} \chi_{E^{\delta}} \big\|_{L^2(\R^n)}^2 = \sum_{E_1, E_2 \in \sE} |E_1^{\delta} \cap E_2^{\delta}|,
\end{equation}
where $\sE$ is a family of axis-parallel ellipsoids with $\delta$-separated centres and principal radii belonging to $[1,2]$. From this we are immediately led to consider the volume $|E_1^{\delta} \cap E_2^{\delta}|$ for pairs of ellipsoids $E_1$, $E_2 \in \cE$. As suggested by our earlier discussion, an unfavourable situation occurs if, generically, $E_1$ and $E_2$ have high order tangency, which results in a large intersection between $E_1^{\delta}$ and $E_2^{\delta}$. However, our parameter counting heuristic suggests that higher order tangency should not be generic behaviour. \medskip

Fix an ellipsoid $E_1$ and let $x_2$ denote a choice of centre for the second ellipsoid. We now consider the set $B_1$ of all $\omega \in E_1$ for which there exists some $\bdr \in [1,2]^n$ such that $E(x_2,\bdr)$ has high order tangency
with $E_1$ at $\omega$. By our earlier parameter counting, we expect the set $B_1$ to be small. In particular, we expect it to be a lower dimensional submanifold of $E_1$. Our strategy is to remove the exceptional set $B_1$ from $E_1$, leaving behind the \textit{refined ellipsoid} $E_1^{\star} := E_1 \setminus B_1$, so that $|E_1^{\delta, \star} \cap E_2^{\delta}|$ satisfies a favourable bound whenever $E_2$ is centred at $x_2$ (here $E_1^{\delta, \star}$ is essentially the $\delta$-neighbourhood of $E_1^{\star}$).\medskip

\noindent \textit{Slicing.} The exceptional set $B_1 = B_1(x_2)$ for $E_1$ is defined in relation to some choice of centre $x_2$. Whilst $B_1(x_2)$ itself may be small, the union of these sets over all centres $x_2 \in \R^n$ is large. However, the key observation which drives the proof is:

\begin{quote}
   For $\ell$ a line through the centre of $E_1$, we can take the \textit{same} exceptional set $B_1(x_2)$ for all $x_2 \in \ell$. 
\end{quote}
This suggests a variable slicing argument. We fix a line $\ell$ through $0$ and consider the ostensibly harder problem of proving uniform bounds for the sliced norms
\begin{equation*}
   \sup_{v \in \R^n} \|M^{\delta}_{\star}f\|_{L^2(\ell + v)}.
\end{equation*}
Here the maximal function $M^{\delta}_{\star}$ is defined with respect to the refined ellipsoids $E^{\star}$, formed by removing the `universal' exceptional set associated to $\ell + v$. Since our operator is local, such a sliced bound easily leads to a true $L^2(\R^n)$ estimate. By translation invariance, we may further assume $v = 0$.\medskip

When we apply the duality argument to this sliced formulation, we are led to consider a variant of \eqref{eq: multiplicity} where now all the ellipsoids in $\sE$ are centred on $\ell$ and each instance of $E^{\delta}$ is replaced with the corresponding refined set $E^{\delta, \star}$. In particular, if $E_1$ and $E_2$ are higher order tangent, then $E_1^{\delta, \star} \cap E_2^{\delta, \star} = \emptyset$, which means that the higher order tangent pairs do not contribute to the estimation of the $L^2$ norm. This leads to a favourable estimate for \eqref{eq: multiplicity}: see Proposition~\ref{prop: multiplicity}.\medskip

\noindent \textit{Covering by refinements.} Finally, it remains to deal with the contribution to the operator which arises from the exceptional set. To do this, we choose a (finite) set $\{\ell_k\}_{k=1}^n$ of lines through $0$ and derive exceptional sets for each $1 \leq k \leq n$. We let $M_k^{\delta}$ denote the maximal operator defined with respect to the refined ellipsoids $E^k := E \setminus B^k$, where $B^k$ is the exceptional set associated to $\ell_k$. Given a point $\omega \in E$, by choosing $\{\ell_k\}_{k = 1}^n$ appropriately it is possible to ensure that there exists some $1 \leq k \leq n$ such that $\omega$ lies in the refinement $E^k $. This allows us to bound the $M^{\delta}$ in terms of the family of maximal operators $M_k^{\delta}$.




\subsection*{Notational conventions} Given a list of objects $L$ and non-negative real numbers $A$, $B$, we write $A \lesssim_L B$ or $A = O_L(B)$ if $A \leq C_L B$ where $C_L > 0$ is a constant depending only on the objects listed in $L$ and a choice of dimension $n$. If $V$ is a vector subspace of $\R^n$, then we let $\mathrm{proj}_V \colon \R^n \to V$ denote the orthogonal projection onto $V$.




\subsection*{Structure of the article} In \S\ref{sec: slicing dual} we carry out the preliminary slicing and duality steps of the argument, and reduce matters to proving volume bounds for intersections between refined ellipsoidal annuli. In \S\ref{sec: vol bounds}, we prove the required volume bounds and thereby complete the proof of Theorem~\ref{thm: discretised}. In \S\ref{sec: Fourier analysis} we combine Theorem~\ref{thm: discretised} with the $L^p$-Sobolev estimates from \cite{LLO} to derive Theorem~\ref{thm: main}. Finally, in the appendix we discuss the details of the modified Knapp example, which demonstrates the sharpness of our theorem in the $n=3$ case.




\subsection*{Thanks} The authors would like to thank David Beltran and Juyoung Lee for comments and corrections to an earlier version of this manuscript.




\subsection*{Acknowledgements} The first author is supported by New Investigator Award UKRI097. The second author is supported by an NSERC Discovery Grant.




\section{Slicing and duality}\label{sec: slicing dual}




\subsection{Discretised operators} 
Recall the definition of the ellipsoid  $E(x; \bdr)$ and the ellipsoidal annulus $E^\delta(x; \bdr)$ from \S\ref{sec: discretised estimate}. In what follows, we will fix a small constant $0 < c_n < 1$, which is chosen depending only on the ambient dimension $n$, subject to the forthcoming requirements of the proof. In order to study $\cM_{\mathrm{st}}$, we consider the localised and discretised operators
\begin{equation}\label{eq: new Mdelta}
    M^{\delta}f(x) := \sup_{\bdr \in \fkR} \fint_{E^{\delta}(x,\bdr)} |f|,
\end{equation}
where $\fkR := \big[1, 1 + c_n^2\big]^n$ is a restricted set of radii. Note that the above definition is a slight abuse of notion, since the operator $M^\delta$ was previously defined in \eqref{eq: defining M delta f}. In this previous definition, $\bdr$ ranged over the domain $[1,2]^n$ rather than $[1, 1 + c_n^2]^n$. This distinction is harmless, since for each $1\leq p\leq\infty$, these two variants of $M^{\delta}$ have comparable $L^p(\R^n)\to L^p(\R^n)$ operator norm. 

Our goal is to prove Theorem~\ref{thm: discretised} with $M^{\delta}$ now defined as in \eqref{eq: new Mdelta}.




\subsection{Removing the exceptional sets} Clearly, provided the dimensional constant $c_n > 0$ is chosen sufficiently small, we may ensure that
\begin{equation}\label{eq: exceptional remove 1}
   \max_{1 \leq k \leq n} |\omega_k|^3 \geq 2c_n > 0 \qquad \textrm{for all $\omega \in \R^n$ with $1/2 \leq |\omega| \leq 2$.}  
\end{equation}
Since the operator $M^{\delta}$ is local and translation-invariant, to prove \eqref{eq: discretised global} it suffices to prove the local estimate
\begin{equation*}
    \|M^{\delta}f\|_{L^2(B(0, 1))} \lesssim_{\varepsilon} \delta^{-\varepsilon} \|f\|_{L^2(\R^n)}.
\end{equation*}

For each $(x, \bdr) \in \R^n \times (0,\infty)^n$, define the affine map 
\begin{equation*}
 A_{x,\bdr} \colon \R^n \to \R^n, \quad    \omega \mapsto x + \bddelta_{\bdr} \,\omega.
\end{equation*}
For all $0 < \delta < 1$, note that $E^{\delta}(x, \bdr)$ is the image of the (spherical) annulus $E^{\delta}(0, \bdone)$ under $A_{x, \bdr}$, where $\bdone = (1, \dots, 1)$. We decompose 
\begin{equation*}
    E^{\delta}(x,\bdr) = \bigcup_{k = 1}^n E^{\delta, k}(x,\bdr) \quad \textrm{where} \quad E^{\delta, k}(x,\bdr) := E^{\delta}(x,\bdr) \cap A_{x,\bdr}\big(\Omega_k\big)
\end{equation*}
for
\begin{equation}\label{eq: Omega k}
    \Omega^k := \big\{ \omega \in \R^n : |\omega_k|^3 \geq 2c_n \big\}.
\end{equation}
Here we use \eqref{eq: exceptional remove 1}, which ensures that the sets $E^{\delta, k}(x, \bdr)$ indeed cover $E^{\delta}(x,\bdr)$. 
Observe that the sets $E^{\delta, k}(x,\bdr)$ are translation-invariant, in the sense that
\begin{equation*}
    E^{\delta, k}(x,\bdr) = x + E^{\delta, k}(0,\bdr) \qquad \textrm{for all $(x,\bdr) \in \R^n \times (0,\infty)^n$.}
\end{equation*}

Using the above, we define 
\begin{equation*}
    M_k^{\delta}f(x) := \sup_{\bdr \in \fkR} \frac{1}{|E^{\delta}(x,\bdr)|} \int_{E^{\delta, k}(x,\bdr)} |f| \qquad \textrm{for $1 \leq k \leq n$,}
\end{equation*}
so that 
\begin{equation}\label{eq: M delta dominated}
    M^{\delta}f(x) \leq \sum_{k = 1}^n M_k^{\delta}f(x).
\end{equation}

We now carry out a preliminary `slicing' of the domain. Define  
\begin{equation}\label{eq: dirs}
 d_k = (d_{k,1}, \dots, d_{k,n}) \in \R^n \quad \textrm{with} \quad   d_{k,j} := 
    \begin{cases}
        1 & \textrm{if $j \neq k$} \\
        0 & \textrm{if $j = k$}
    \end{cases}
    \qquad \textrm{for $1 \leq k \leq n$,}
\end{equation}
let $\ell_k$ denote the $1$-dimensional linear subspace of $\R^n$ spanned by $d_k$, and let $V_k := \ell_k^{\perp}$ denote the codimension $1$ subspace normal to $d_k$. By the Fubini--Tonelli theorem and \eqref{eq: M delta dominated}, we have
\begin{align*}
  \|M^{\delta}f\|_{L^2(B(0, 1))} &\leq \sum_{k = 1}^n \Big( \int_{V_k \cap B(0, 1)} \|M^{\delta}_kf\|_{L^2(\ell_k(v_k))}^2\,\ud v_k \Big)^{1/2} \\
  &\lesssim \max_{1 \leq k \leq n} \sup_{v_k \in V_k \cap B(0, 1)} \|M_k^{\delta}f\|_{L^2(\ell_k(v_k))},
\end{align*}
where $\ell_k(v_k) := \{td_k + v_k : t \in [-1,1]\} \subseteq \ell_k + v_k$ is a line segment.




\subsection{Duality and reduction to volume bounds}
We now apply a standard duality argument, and exploit the translation-invariance of the problem, to reduce matters to the following multiplicity bound.

\begin{proposition}\label{prop: multiplicity} For $1 \leq k \leq n$ and $0 < \delta < 1$, we have
\begin{equation*}
    \big\| \sum_{E \in \sE} \chi_{E^{\delta, k}} \big\|_{L^2(\R^n)} \lesssim \log \delta^{-1} \delta^{1/2} \big[ \# \sE\big]^{1/2}
\end{equation*}
whenever $\sE$ is a family of axis-parallel ellipsoids with the following properties:
\begin{enumerate}[i)]
    \item The centres are $\delta$-separated and lie on $\ell_k(0)$;
    \item The principal radii belong to $\fkR$.
\end{enumerate}
\end{proposition}

The key ingredient in the proof of Proposition~\ref{prop: multiplicity} is the following volume bound.

\begin{lemma}\label{lem: volume} Let $1 \leq k \leq n$ and $\delta > 0$. Given $x_j := t_jd_k$ for $t_j \in [-1, 1]$ and $\bdr_j \in \fkR$ for $j = 1$, $2$, we have
\begin{equation*}
    \big|E^{\delta, k}(x_1, \bdr_1) \cap E^{\delta, k}(x_2, \bdr_2)| \lesssim \log \delta^{-1} \frac{\delta^2}{\delta + |t_1 - t_2|}. 
\end{equation*}
\end{lemma}

We stress that the volume bound involves the sets $E^{\delta, k}$, which are each formed by removing some exceptional set from the parent annulus $E^{\delta}$. It is crucial that the exceptional set is removed to avoid degenerate behaviour arising from high order tangencies between the parent annuli. 

Once Lemma~\ref{lem: volume} is established, the proof of Proposition~\ref{prop: multiplicity} follows from the standard $L^2$ Kakeya argument of C\'ordoba~\cite{Cordoba1977} (see also \cite{Carbery1988} or \cite[Chapter 22]{Mattila_book}).

\begin{proof}[Proof of Proposition~\ref{prop: multiplicity}] Fixing $\sE$ a family of ellipsoids as in the statement of the proposition and expanding out the square in the $L^2$ norm, we have
\begin{equation}\label{eq: cordoba 1}
    \big\| \sum_{E \in \sE} \chi_{E^{\delta, k}} \big\|_{L^2(\R^n)}^2 = \sum_{E_1 \in \sE} \sum_{E_2 \in \sE} |E_1^{\delta, k} \cap E_2^{\delta, k}|.
\end{equation}

For $E_1 \in \sE$, we dyadically partition the sum over $E_2 \in \sE$ in \eqref{eq: cordoba 1} according to the distance between the centres. In particular, define
\begin{align*}
  \sE^{E_1}_{< \delta} &:= \big\{  E_2 \in \sE : |x(E_2) - x(E_1)| < \delta \big\}, \\
  \sE^{E_1}_{\tau} &:= \big\{E_2 \in \sE : \tau/2 \leq |x(E_2) - x(E_1)| < \tau\big\}
\end{align*}
for $\delta \leq \tau \leq 1$, where $x(E)$ denotes the centre of an ellipsoid $E$. Since, by hypothesis, the centres $\{x(E) : E \in \sE\}$ form a $\delta$-separated subset of the line segment $\ell_k(0)$, we have
\begin{equation*}
    \big[\#\sE^{E_1}_{< \delta}\big] \lesssim 1 \qquad \textrm{and} \qquad  \big[\#\sE^{E_1}_{\tau}\big] \lesssim \delta^{-1}\tau \qquad \textrm{for all $\delta \leq \tau \leq 1$ and all $E_1 \in \sE$.}
\end{equation*}
Combining these observations with the volume bound from Lemma~\ref{lem: volume}, 
\begin{align}
    \nonumber
   \sum_{E_2 \in \sE} |E_1^{\delta, k} \cap E_2^{\delta, k}| &\leq  \sum_{E_2 \in \sE^{E_1}_{< \delta}} |E_1^{\delta, k} \cap E_2^{\delta, k}| + \sum_{\substack{\delta \leq \tau \leq 1 \\ \mathrm{dyadic}}} \sum_{E_2 \in \sE^{E_1}_{\tau}} |E_1^{\delta, k} \cap E_2^{\delta, k}| \\
   \nonumber
   & \lesssim \log \delta^{-1} \delta^2\Big( \delta^{-1} \big[\#\sE^{E_1}_{< \delta}\big] + \sum_{\substack{\delta \leq \tau \leq 1 \\ \mathrm{dyadic}}} \tau^{-1} \big[\#\sE^{E_1}_{\tau}\big] \Big) \\
   \label{eq: cordoba 2}
   & \lesssim \big(\log \delta^{-1}\big)^2 \delta.
\end{align}

To conclude the proof, we apply the estimate \eqref{eq: cordoba 2} to \eqref{eq: cordoba 1} for each $E_1 \in \sE$ to deduce that
\begin{equation*}
    \big\| \sum_{E \in \sE} \chi_{E^{\delta, k}} \big\|_{L^2(\R^n)}^2 \lesssim \big(\log \delta^{-1}\big)^2 \delta \big[\# \sE\big].
\end{equation*}
Taking square roots, this implies the desired inequality. 
\end{proof}

It remains to prove the volume bound in Lemma~\ref{lem: volume}. 




\section{Volume bounds for intersecting ellipsoidal annuli}\label{sec: vol bounds}




\subsection{Controlling the tangency set} Recall that $\bdone := (1, \dots, 1)$ and $E(0, \bdone) = S^{n-1}$ is the unit sphere in $\R^n$. Fixing $x \in \R^n$, $\bdr \in [1/2,2]^n$ and $0 < \delta < 1$, consider the intersection 
\begin{align*}
 E^{\delta}(0, \bdone) \cap E^{\delta}(x, \bdr)& = \big\{\omega \in \R^n : |F_{0, \bdone}(\omega)| < \delta \textrm{ and }   |F_{x, \bdr}(\omega)| < \delta \big\} \\
 &= \big(F_{0, \bdone}, F_{x, \bdr}\big)^{-1}\big([-\delta, \delta]^2\big). 
\end{align*}
For $u = (u_1, u_2) \in [-1,1]^2$ let 
\begin{equation*}
    \gamma_{x, \bdr}(u) :=  (F_{0, \bdone}, F_{x, \bdr})^{-1}(\{u\})  = \{\omega \in \R^n : F_{0, \bdone}(\omega) =  u_1 \textrm{ and } F_{x, \bdr}(\omega) =  u_2 \big\}
\end{equation*}
denote the fibre of the map $(F_{0, \bdone}, F_{x, \bdr}) \colon \R^n \to \R^2$. 

\begin{lemma}\label{lem: fibre length} For $x \in \R^n \setminus \{0\}$, $u \in [-1,1]^2$ and $0 < \rho < 1$, we have
\begin{equation*}
    \sup_{\xi \in \R^n} \cH^{n-2}(\gamma_{x, \bdr}(u) \cap B(\xi, \rho)) \lesssim \rho^{n-2}.
\end{equation*}
\end{lemma}
Lemma \ref{lem: fibre length} is an immediate consequence of the following standard estimate (see, for instance, \cite{Wongkew1993}). Let $X\subset\R^n$ be a $d$-dimensional algebraic variety. Then for all $\xi\in\R^n$ and $\rho>0$, we have
\begin{equation*}
    \cH^{d}(X \cap B(\xi, \rho))\lesssim \operatorname{degree}(X) \rho^d.
\end{equation*}
In Lemma \ref{lem: fibre length}, the set $\gamma_{x, \bdr}(u)$ is the proper intersection of two ellipsoids in $\R^n$, and thus is an $(n-2)$-dimensional variety of degree at most 4.

Continuing with the above setup, we decompose the sets $E^{\delta, k}(0, \bdone)$ for $1 \leq k \leq n$ according to the size of the angle between the normal vectors $\nabla F_{0,\bdone}(\omega)$ and $\nabla F_{x,\bdr}(\omega)$. To this end, define the $2 \times n$ Jacobian matrix
\begin{equation*}
   J_{x, \bdr}(\omega)  := 
    \begin{bmatrix}
        \frac{\partial F_{0,\bdone}}{\partial \omega_1}(\omega) & \dots & \frac{\partial F_{0,\bdone}}{\partial \omega_n}(\omega) \\[2pt]
        \frac{\partial F_{x ,\bdr}}{\partial \omega_1}(\omega) & \dots & \frac{\partial F_{x ,\bdr}}{\partial \omega_n}(\omega)
    \end{bmatrix}
\end{equation*}
and set
\begin{align*}
   \|J_{x, \bdr}(\omega)\| &:=  \big(\det J_{x, \bdr}(\omega) J_{x, \bdr}(\omega)^{\top}\big)^{1/2} \\
   &= \big(|\nabla F_{0,\bdone}(\omega)|^2|\nabla F_{x,\bdr}(\omega)|^2 - \inn{\nabla F_{0,\bdone}(\omega)}{\nabla F_{x,\bdr}(\omega)}^2 \big)^{1/2}.
\end{align*}
Note, in particular, that the above quantity vanishes if and only if $\nabla F_{0,\bdone}(\omega)$ and $\nabla F_{x,\bdr}(\omega)$ are parallel. We consider the sets
\begin{equation*}
 E_{x, \bdr, < \rho}^{\delta, k}(0, \bdone) :=  \big\{\omega \in E^{\delta, k}(0, \bdone) : \|J_{x, \bdr}(\omega)\| < \rho \big\} \quad \textrm{for $\rho > 0$.}
\end{equation*}
With these definitions, the key result is as follows.

\begin{lemma}\label{lem: near root} Provided $c_n > 0$ is chosen sufficiently small, there exists a dimensional constant $C_n \geq 1$ such that the following holds. Let $1 \leq k \leq n$ and $x := t \tilde{d}_k$ for some $t \in (0,2]$ and $\tilde{d}_k \in \R^n$ satisfying $|\tilde{d}_k - d_k|_{\infty} < c_n^2$ and $\bdr \in [1/2, 2]^n$. For $0 < \delta \leq \rho \leq 1$, there exists some set $\Xi^k(t, \bdr; \rho) \subset \R^n$ satisfying
\begin{equation*}
 \# \Xi^k(t, \bdr; \rho)  \lesssim 1 \qquad \textrm{and} \qquad  E_{x, \bdr, < \rho}^{\delta, k}(0, \bdone) \subseteq \bigcup_{\xi \in \Xi^k(t, \bdr; \rho)} B(\xi, C_n\rho/t) \cup N, 
\end{equation*}
where $N \subseteq \R^n$ is a Lebesgue null set.
\end{lemma}

Once Lemma~\ref{lem: near root} is established, the desired volume bound in Lemma~\ref{lem: volume} easily follows.

\begin{proof}[Proof of Lemma~\ref{lem: volume}] Write $\bdr_1 := (r_{1,1}, \dots, r_{1, n})$ and let $\tilde{d}_k := (\delta_{\bdr_1})^{-1}d_k$, so that
\begin{equation*}
  |\tilde{d}_k - d_k|_{\infty} = \max_{1 \leq j \leq n}|\tilde{d}_{k,j} - d_{k,j}| = \max_{1 \leq j \leq n} \frac{|1 - r_{1,j}|}{r_{1,j}} |d_{k,j}| \leq c_n^2,
\end{equation*}
where the last step relies on the restriction $\bdr_1 \in \fkR$. By applying an affine transformation, we may assume $(x_1, \bdr_1) = (0, \bdone)$, provided we work with $(x, \bdr) := (x_2, \bdr_2)$ satisfying $x = t \tilde{d}_k$ for some $t \in [0, 2]$ and $\bdr \in [1/2, 2]^n$. We may also assume $t > 10 \delta$, since otherwise the bound is trivial. 

Consider the subsets of $E^{\delta, k}(0, \bdone)$ given by
\begin{equation*}
    E_{x, \bdr,  \tang}^{\delta, k}(0, \bdone) := E_{x, \bdr, < 2(t\delta)^{1/2}}^{\delta, k}(0, \bdone), \quad E_{x, \bdr, \trans}^{\delta, k}(0, \bdone) := E^{\delta, k}(0, \bdone) \setminus E_{x, \bdr, < t}^{\delta, k}(0, \bdone)
\end{equation*}
and
\begin{equation*}
    E_{x, \bdr, \rho}^{\delta, k}(0, \bdone) := E_{x, \bdr, < 2\rho}^{\delta, k}(0, \bdone) \setminus E_{x, \bdr, < \rho}^{\delta, k}(0, \bdone)
\end{equation*}
so that
\begin{equation*}
  E^{\delta, k}(0, \bdone)  = E_{x, \bdr, \tang}^{\delta, k}(0, \bdone) \cup E_{x, \bdr, \trans}^{\delta, k}(0, \bdone) \cup \bigcup_{\substack{(t\delta)^{1/2} \leq \rho \leq t \\ \mathrm{dyadic}}} E_{x, \bdr, \rho}^{\delta, k}(0, \bdone).
\end{equation*}

On the one hand, Lemma~\ref{lem: near root} implies that
\begin{equation*}
   | E_{x, \bdr, \tang}^{\delta, k}(0, \bdone) \cap E^{\delta}(x, \bdr)| \lesssim \max_{\xi \in \R^n} |E^{\delta}(0, \bdone) \cap B(\xi, 2C_n(\delta/t)^{1/2})| \lesssim (\delta/t)^{(n-3)/2}(\delta^2/t).   
\end{equation*}
On the other hand, by the coarea formula (see, for instance, \cite[Theorem 3.2.12, p.249]{Federer1969}), 
\begin{equation*}
    | E_{x, \bdr, \trans}^{\delta, k}(0, \bdone) \cap E^{\delta}(x, \bdr)| = \int_{[-\delta, \delta]^2}  \int_{\gamma_{x, \bdr}(u) \cap E_{x, \bdr, \trans}^{\delta, k}(0, \bdone)} \frac{\ud \cH^{n-2}(\omega)}{\|J_{x, \bdr}(\omega)\|}\,\ud u.
\end{equation*}
Since $\|J_{x,\bdr}(\omega)\| \geq t$ for $\omega \in E_{x, \bdr, \trans}^{\delta, k}(0, \bdone)$ and $\cH^{n-2}(\gamma_{x,\bdr}(u)) \lesssim 1$ for all $u \in [-\delta, \delta]^2$, it follows that
\begin{equation*}
    | E_{x, \bdr, \trans}^{\delta, k}(0, \bdone) \cap E^{\delta}(x, \bdr)| \lesssim \delta^2/t. 
\end{equation*}

Finally, given $(t\delta)^{1/2} \leq \rho \leq t$ dyadic, the coarea formula similarly gives 
\begin{equation*}
    | E_{x, \bdr, \rho}^{\delta, k}(0, \bdone) \cap E^{\delta}(x, \bdr)| = \int_{[-\delta, \delta]^2}  \int_{\gamma_{x, \bdr}(u) \cap E_{x, \bdr, \rho}^{\delta, k}(0, \bdone) \setminus N} \frac{\ud \cH^{n-2}(\omega)}{\|J_{x, \bdr}(\omega)\|}\,\ud u,
\end{equation*}
where $N$ is the Lebesgue null set featured in Lemma~\ref{lem: near root}. Recall that $\|J_{x,\bdr}(\omega)\| \geq \rho$ for $\omega \in E_{x, \bdr, \rho}^{\delta, k}(0, \bdone)$ and $E_{x, \bdr, \rho}^{\delta, k}(0, \bdone) \subseteq E_{x, \bdr, < 2\rho}^{\delta, k}(0, \bdone)$. Furthermore, by Lemma~\ref{lem: near root}, given any $u \in [-\delta, \delta]^2$ we have
\begin{align*}
    \cH^{n-2}\big(\gamma_{x, \bdr}(u) \cap E_{x, \bdr, < 2\rho}^{\delta, k}(0, \bdone)\setminus N \big) &\lesssim \max_{\xi \in \R^n} \cH^{n-2}\big(\gamma_{x, \bdr}(u) \cap B(\xi, 2C_n\rho/t) \big) \\
    &\lesssim (\rho/t)^{n-2},
\end{align*}
where the last step is by Lemma~\ref{lem: fibre length}. Thus, altogether we have
\begin{equation*}
    | E_{x, \bdr, \rho}^{\delta, k}(0, \bdone) \cap E^{\delta}(x, \bdr)| \lesssim (\rho/t)^{n-3}(\delta^2/t).
\end{equation*}

Combining the preceding bounds, 
\begin{equation*}
   | E^{\delta, k}(0, \bdone) \cap E^{\delta, k}(x, \bdr)| \lesssim (\delta^2/t)\sum_{\substack{(t\delta)^{1/2} \leq \rho \leq t \\ \mathrm{dyadic}}} (\rho/t)^{n-3}  \lesssim \log \delta^{-1} (\delta^2/t),
\end{equation*}
as required. We remark that the logarithmic factor is only incurred when $n = 3$.  
\end{proof}

\subsection{Proof of Lemma~\ref{lem: near root}} Let $1 \leq k \leq n$, $x = t\tilde{d}_k$ for some $t \in (0, 2]$ and $\tilde{d}_k \in \R^n$ satisfying $|\tilde{d}_k - d_k|_{\infty} < c_n^2$ and $\bdr \in [1/2, 2]^n$. Our first step towards Lemma~\ref{lem: near root} is to study the vanishing set of the function $\omega \mapsto \|J_{x, \bdr}(\omega)\|$. By the Cauchy--Binet formula, 
\begin{equation}\label{eq: Cauchy--Binet}
  \|J_{x, \bdr}(\omega)\|^2 =  \sum_{1 \leq i < j \leq n} \det
    \begin{vmatrix}
        \frac{\partial F_{0, \bdone}}{\partial \omega_i}(\omega) & \frac{\partial F_{0, \bdone}}{\partial \omega_j}(\omega) \\
        \frac{\partial F_{x, \bdr}}{\partial \omega_i}(\omega) & \frac{\partial F_{x, \bdr}}{\partial \omega_j}(\omega)
    \end{vmatrix}^2 = 16\sum_{1 \leq i < j \leq n} G_{i,j}^k(\omega, t, \bdr)^2
\end{equation}
where
\begin{equation*}
    G_{i,j}^k(\omega, t, \bdr) := 
    \Big(\frac{1}{r_j^2} - \frac{1}{r_i^2}\Big)\omega_i\omega_j - t\Big(\frac{\tilde{d}_{k,j}\omega_i}{r_j^2}- \frac{\tilde{d}_{k,i}\omega_j}{r_i^2}\Big)
    \end{equation*}
for $1 \leq i < j \leq n$. Let $G_j^k := G_{j,k}^k$ for $1 \leq j \leq n$ and observe that
\begin{equation}\label{eq: degenerate G}
  \omega_k G_{i,j}^k(\omega, t, \bdr) =  \omega_j G_i^k(\omega, t, \bdr) - \omega_i G_j^k(\omega, t, \bdr)  \quad \textrm{for $1 \leq i < j \leq n$ with $i$, $j \neq k$.}
\end{equation}

Consider the domain 
\begin{equation}\label{eq: X^k}
    X^k := \Omega^k \times (0,2] \times [1/2,2]^n 
\end{equation}
where $\Omega^k$ is defined in \eqref{eq: Omega k}, and the map $\Phi^k \colon X^k \to \R^n$ given by
\begin{equation}\label{eq: defn Phi k}
\Phi^k(\omega, t, \bdr)  := \Big(G_1^k (\omega, t, \bdr), \dots, G_{k-1}^k (\omega, t, \bdr), G_{k+1}^k (\omega, t, \bdr), \dots, G_n^k (\omega, t, \bdr), \frac{F_{0,\bdone}(\omega)}{2} \Big). 
\end{equation}
If $(\omega, t, \bdr) \in X^k$ satisfies $\Phi^k (\omega,t,\bdr) = 0$, then $\omega \in E(0,\bdone)$ and, by \eqref{eq: degenerate G}, we have $\omega_k G_{i,j}^k(\omega, t, \bdr) = 0$ for all $1 \leq i < j \leq n$. If $\omega_k = 0$, then the condition $G_j^k(\omega, t, \bdr) = 0$ implies that $\omega_j = 0$ for $1 \leq j \leq n$ with $j \neq k$ and so $\omega = 0$, contradicting $\omega \in E(0, \bdone)$. Hence, we must have $\omega_k \neq 0$ and so $G_{i,j}^k(\omega, t, \bdr) = 0$ for all $1 \leq i < j \leq n$. By the definition of the $G_{i,j}^k$ functions, this means there exists some $\lambda \in \R$ such that 
\begin{equation*}
    \nabla F_{x, \bdr}(\omega) = \lambda \nabla F_{0, \bdone}(\omega).
\end{equation*}

\begin{lemma}\label{lem: non deg} There exists a dimensional constant $0 < \bar{c}_n < 1$ such that if $1 \leq k \leq n$ and $(\omega, t, \bdr) \in X^k$ with $t > 0$ satisfies $|\Phi^k(\omega, t, \bdr)| < \bar{c}_n t$, then the following hold:
\begin{enumerate}[i)]
    \item $|\det D_{\omega} \Phi^k(\omega,t,\bdr)| \prod_{j=1}^n |\omega_j| \gtrsim t^{n-1}$; 
    \item $|D_{\omega} \Phi^k(\omega,t,\bdr)^{-1}| \lesssim t^{-1}$. 
\end{enumerate} 
\end{lemma}

Here $D_{\omega} \Phi^k$ denotes the Jacobian of $\Phi^k$ with respect to the $\omega$ variables; that is,
\begin{equation*}
    D_{\omega} \Phi^k(\omega, t, \bdr) := \big(\nabla_{\omega}G_1^k (\omega, t, \bdr), \dots, \nabla_{\omega}G_n^k (\omega, t, \bdr), \nabla_{\omega}F_{0,\bdone}(\omega)/2 \big), 
\end{equation*}
where $\nabla_{\omega}G_k^k (\omega, t, \bdr)$ is omitted. The statement in part ii) is interpreted as a bound on the $\ell^2$ norm of the components of the inverse matrix $D_{\omega} \Phi^k(\omega,t,\bdr)^{-1}$.

\begin{proof}[Proof of Lemma~\ref{lem: non deg}] Let  $(\omega, t, \bdr) \in X^k$ with $t > 0$ and initially suppose the weaker hypothesis $|\Phi^k(\omega, t, \bdr)| < t$. Fix $1 \leq j \leq n$ with $j \neq k$ and note that 
\begin{subequations}
\begin{equation}\label{eq: transversal 2 a}
    \partial_{\omega_i} G_j^k(\omega, t, \bdr) = 0 \qquad \textrm{unless $i \in \{j, k\}$.}
\end{equation}
On the other hand, 
\begin{align*}
    \partial_{\omega_j} G_j^k(\omega, t, \bdr) &= 
        \Big(\frac{1}{r_k^2} - \frac{1}{r_j^2}\Big)\omega_k - \frac{t \tilde{d}_{k,k}}{r_k^2}, \\
        \partial_{\omega_k} G_j^k(\omega, t, \bdr) &= 
        \Big(\frac{1}{r_k^2} - \frac{1}{r_j^2}\Big)\omega_j + \frac{t \tilde{d}_{k,j}}{r_j^2},
\end{align*}
so that 
\begin{align}
\label{eq: transversal 2 b}
  \omega_j\partial_{\omega_j} G_j^k(\omega, t, \bdr) &= 
        G_j^k(\omega, t, \bdr) -  \frac{t \tilde{d}_{k,j}\omega_k}{r_j^2},  \\
\label{eq: transversal 2 c}
\omega_k\partial_{\omega_k} G_j^k(\omega, t, \bdr) &= 
        G_j^k(\omega, t, \bdr) + \frac{t \tilde{d}_{k,k}\omega_j}{r_k^2}.    
\end{align}
\end{subequations}
Using these formul\ae, we may write
\begin{equation*}
    \Big(\prod_{j=1}^n r_j^2 \omega_j \Big) \det D_{\omega} \Phi^k(\omega, t, \bdr) = \det \big[A^k(\omega, t,\bdr) +  B^k(\omega, t,\bdr)\big]
\end{equation*}
where $A^k(\omega, t,\bdr)$ is given by
\begin{equation*}
\begin{bmatrix}
       -  t \tilde{d}_{k,1}\omega_k & \cdots & 0 & t\tilde{d}_{k,k} \omega_1 & 0 & \cdots & 0 \\
      \vdots & \ddots & \vdots & \vdots & \vdots & \ddots & \vdots  \\
0 & \cdots & -  t\tilde{d}_{k,k-1}\omega_k & t\tilde{d}_{k,k} \omega_{k-1} & 0 & \cdots & 0 \\
0 & \cdots & 0 & t\tilde{d}_{k,k} \omega_{k+1} & -  t\tilde{d}_{k,k+1}\omega_k &  \cdots & 0  \\
\vdots & \ddots & \vdots & \vdots & \vdots & \ddots & \vdots \\
0 & \cdots & 0 & t\tilde{d}_{k,k} \omega_n & 0 & \cdots & -t\tilde{d}_{k,n}\omega_k \\
r_1^2\omega_1^2 & \cdots & r_{k-1}^2\omega_{k-1}^2 &  r_k^2\omega_k^2 & r_{k+1}^2\omega_{k+1}^2  &  \cdots  & r_n^2\omega_n^2
    \end{bmatrix}
\end{equation*}
and, writing $G_j^k$ for $G_j^k(\omega, t, \bdr)$, the matrix $B^k(\omega, t,\bdr)$ is given by
\begin{equation*}
\begin{bmatrix}
       r_1^2G_1^k & \cdots & 0 & r_k^2G_1^k& 0 & \cdots & 0  \\
      \vdots & \ddots & \vdots & \vdots & \vdots & \ddots & \vdots \\
      0 & \cdots & r_{k-1}^2G_{k-1}^k & r_k^2G_{k-1}^k& 0 & \cdots & 0 & \\
      0 & \cdots & 0 & 
  r_k^2G_{k+1}^k & r_{k+1}^2G_{k+1}^k & \cdots & 0  \\
      \vdots & \ddots & \vdots & \vdots & \vdots & \ddots & \vdots  \\
      0 & \cdots & 0 & r_k^2G_n^k &0 & \cdots & r_n^2G_n^k \\
    0 & \cdots & 0 & 0 &0 & \cdots & 0 
    \end{bmatrix}.
\end{equation*}
Thus, by the multi-linearity of the determinant and the hypothesis $|\Phi^k(\omega, t, \bdr)| < t$, we deduce that 
\begin{equation}\label{eq: transversal 3}
    \Big(\prod_{j=1}^n r_j^2 \omega_j \Big) \det D_{\omega} \Phi^k(\omega, t, \bdr) = \det A^k(\omega, t,\bdr) + \cE^k(\omega, t, \bdr). 
\end{equation}
where the function $\cE^k \colon X^k \to \R$ satisfies
\begin{equation}\label{eq: bdCEk}
  |\cE^k(\omega, t, \bdr)| \lesssim t^{n-2}|\Phi^k(\omega, t, \bdr)| \qquad \textrm{for all $(\omega, t, \bdr) \in X^k$.}
\end{equation}

Recall that, given $\ell, m \in \N$ and $W \in \mathrm{Mat}(\bbF, \ell \times \ell)$, $X \in \mathrm{Mat}(\bbF, \ell \times m)$, $Y \in \mathrm{Mat}(\bbF,  m \times \ell)$ and $Z \in \mathrm{Mat}(\bbF, m \times m)$ matrices over a field $\bbF$ with $\det W \neq 0$, we have the Schur complement identity\footnote{This follows by factoring 
    \begin{equation*}
  \begin{bmatrix} 
 W & X \\
 Y & Z 
\end{bmatrix} = 
\begin{bmatrix} 
I & 0 \\
YW^{-1} & I \end{bmatrix}
\begin{bmatrix} 
W & 0 \\ 
0 & Z - YW^{-1}X 
\end{bmatrix}
\begin{bmatrix} 
I & W^{-1}X \\ 
0 & I 
\end{bmatrix}.  
\end{equation*}}
\begin{equation*}
    \det
    \begin{bmatrix}
        W & X \\
        Y & Z
    \end{bmatrix}
    = \det W \det (Z - Y W^{-1} X).
\end{equation*}
Rearranging the columns of $A^k(\omega, t, \bdr)$ and taking $Z = (r_k^2\omega_k^2)$ to be a $1 \times 1$ matrix, we can use this formula to deduce that
\begin{equation*}
    \det A^k(\omega, t,\bdr) = (-1)^{k-1} t^{n-1} \omega_k^{n-2} \sum_{j=1}^n \Big(\prod_{\substack{1 \leq i \leq n \\ i \neq j}} \tilde{d}_{k,i}\Big) r_j^2 \omega_j^3.
\end{equation*}
By our choice of vectors $d_k$, as given in \eqref{eq: dirs}, we have 
\begin{equation*}
   \prod_{\substack{1 \leq i \leq n \\ i \neq j}} d_{k,i} = 
   \begin{cases}
       1 & \textrm{if $j = k$,} \\
       0 & \textrm{otherwise.}
   \end{cases}
\end{equation*}
Recall that $|\tilde{d}_k - d_k| < c_n^2$, whilst the hypothesis $|\Phi^k(\omega, t, \bdr)| < t$ implies $|F_{0, \bdone}(\omega)| \leq 2t$ which in turn implies $|\omega| \lesssim 1$. Thus, we may write
\begin{equation*}
    \det A^k(\omega, t,\bdr) = (-1)^{k-1} t^{n-1} \omega_k^{n-2} \big( r_k^2\omega_k^3 + \cE^k(\omega; \bdr)\big)
\end{equation*}
where $|\cE^k(\omega; \bdr)| \lesssim  c_n^2$. Since $(\omega, t, \bdr) \in X^k$, we have $|\omega_k|^3 \geq 2c_n$. Consequently, provided $c_n > 0$ is chosen sufficiently small depending only on $n$, we may ensure that $|r_k^2 \omega_k^3 + \cE^k(\omega; \bdr)| \geq c_n$ and therefore $|\det A^k(\omega, t,\bdr)| \gtrsim t^{n-1}$. Comparing this with \eqref{eq: transversal 3} and using \eqref{eq: bdCEk}, we conclude that if $\bar c_n$ is selected appropriately, then the hypothesis that $|\Phi^k(\omega, t, \bdr)| < \bar{c}_n t$ leads to the conclusion of part i).\medskip

Turning to part ii), for $1 \leq \alpha, \beta \leq n$ let $M_{\alpha, \beta}^k(\omega, t, \bdr)$ denote the $(n-1) \times (n-1)$ submatrix formed by removing row $\alpha$ and column $\beta$ from $D_{\omega} \Phi^k(\omega, t, \bdr)$. In light of the result from part i) and the adjugate matrix formula for the inverse, to bound the components of the inverse matrix it suffices to show 
\begin{equation}\label{eq: transversal 7}
   |\det M_{\alpha, \beta}^k(\omega, t, \bdr)| \prod_{j=1}^n |\omega_j| \lesssim  t^{n-2} \qquad \textrm{for all $1 \leq \alpha, \beta \leq n$.}
\end{equation}
Note that $\omega_j \neq 0$ for $1 \leq j \leq n$ as a consequence of part i),  and so we are able to divide out the quantity $\prod_{j=1}^n |\omega_j|$.

Fixing the indices $\alpha$, $\beta$ and letting $v_j(\omega, t, \bdr)$ denote the $j$th row of $ M_{\alpha, \beta}^k(\omega, t, \bdr)$ for $1 \leq j \leq n-1$, it follows from \eqref{eq: transversal 2 a}, \eqref{eq: transversal 2 b} and \eqref{eq: transversal 2 c}, together with the hypothesis $|\Phi^k(\omega, t, \bdr)| < ct$, that 
\begin{equation*}
    |v_j(\omega, t, \bdr)||\omega_j| \lesssim t \quad \textrm{for all $1 \leq j \leq n-2$} \quad \textrm{and} \quad  |v_{n-1}(\omega, t, \bdr)| \lesssim 1. 
\end{equation*}
Thus, we obtain \eqref{eq: transversal 7} as a consequence of, say, Hadamard's inequality. 
\end{proof}

Given $(t, \bdr) \in (0,2] \times [1/2,2]^n$ and $\rho > 0$, consider the set 
\begin{equation}\label{eq: good omega set 2}
\Omega^k(t, \bdr; \rho) := \big\{ \omega \in \R^n : (\omega, t, \bdr) \in X^k \textrm{ and } |\Phi(\omega, t, \bdr)| < \rho \big\}.
\end{equation}

\begin{lemma}\label{lem: comp diam} Let $1 \leq k \leq n$, $(t, \bdr) \in (0,2] \times [1/2,2]^n$ and $0 < \rho < \bar{c}_n t$ where $\bar{c}_n >0$ is the constant appearing in Lemma~\ref{lem: non deg}. There exists some $M \in \N$ with $M = O(1)$ such that
\begin{equation}\label{eq: comp diam}
    \Omega^k(t, \bdr; \rho) = \bigcup_{m = 1}^M \Omega_m^k(t, \bdr; \rho) \cup N
\end{equation}
where the $\Omega_m^k(t, \bdr; \rho)$ are measurable sets satisfying  $\diam \Omega_m^k(t, \bdr; \rho) \lesssim \rho/t$ and $N \subseteq \R^n$ is Lebesgue null.
\end{lemma}

For the proof of Lemma~\ref{lem: comp diam}, it is convenient to appeal to the following algebraic variant of the inverse function theorem. 

\begin{lemma}[{Effective inverse function theorem \cite[Lemma 3.2]{CRW2003}}]\label{lem: effective inverse} For $1 \leq j \leq n$, let $P_j \in \R[X_1, \dots, X_n]$  and consider the mapping $P = (P_1, \dots, P_n) \colon \R^n \to \R^n$. Suppose that the Jacobian determinant $\det D P$ is not identically zero. Then there exists a number $1 \leq M \leq M'$ where $M' := \prod_{j=1}^n \deg P_j$ and a decomposition
\begin{equation*}
    \R^n = N \sqcup \bigsqcup_{m=1}^M U_m
\end{equation*}
with the following properties:
\begin{enumerate}[i)]
    \item For $1\leq m\leq M$, the restriction of $P$ to $U_m$ is a diffeomorphism onto its image;
    \item $N$ is the zero-set of a polynomial of degree $O_{M'}(1)$;
    \item The sets $U_1,\ldots,U_M$ are unions of connected components of $\R^n\backslash N$.
    \end{enumerate}
\end{lemma}

Note that Conclusions ii) and iii) are not stated explicitly in \cite[Lemma 3.2]{CRW2003}, but these are an immediate consequence of the proof in \cite{CRW2003}. Lemma~\ref{lem: effective inverse} is a quantitative version of the inverse function theorem, since it provides control over the number of open sets $M$. 

Roughly speaking, the sets $\Omega_m^k$ from Lemma \ref{lem: comp diam} will be constructed by intersecting $\Omega^k(t, \bdr; \rho)$ with the sets $U_m$ output from Lemma \ref{lem: effective inverse} with $P(\omega) = \Phi^k(\omega, t,\bdr)$ as defined in \eqref{eq: defn Phi k}. However, the sets constructed in this fashion might fail to be connected. We will fix this by recalling a result which says that after further sub-dividing the sets if necessary, we can ensure that these sets are path connected in a certain quantitative sense: the intrinsic and extrinsic metrics on each set are Lipschitz equivalent. We will first recall several definitions.

\begin{definition}
Let $M\geq 1$. A set $X \subset\mathbb{R}^n$ is called a \emph{semi-algebraic set} of complexity at most $M$ if there exists an integer $N\leq M$; polynomials $P_1,\ldots,P_N$, each of degree at most $M$; and a Boolean formula $\Psi\colon \{0,1\}^N\to\{0,1\}$ such that
\[
X = \big\{x\in\mathbb{R}^n \colon \Psi\big( P_1(x)\geq 0, \ldots, P_N(x)\geq 0 \big)=1\big\}.
\]
\end{definition}
We refer the reader to the textbooks by Bochnak, Coste and Roy \cite{BCR} and Basu, Pollack and Roy \cite{BPR} for an introduction to this topic.

A theorem of Milnor and Thom \cite{Milnor, Thom} says that if $P\colon\R^n\to\R$ is a polynomial of degree $M$, then the open set $\R^n\backslash\{P=0\}$ is a union of at most $M(2M-1)^{n-1}$ (Euclidean) connected components. In particular, each of the open sets $U_1,\ldots,U_M$ from Lemma \ref{lem: effective inverse} is a union of $O_{M'}(1)$ connected components of the open set $\R^n\backslash N$, where $M':= \prod_{j=1}^n \deg P_j$.

A theorem of \L{}ojasiewicz (see, for instance, \cite[Theorem 5.22]{BPR}) says that if $X\subseteq \mathbb{R}^n$ is a semi-algebraic set of complexity at most $M$, then each Euclidean connected component of $X$ is semi-algebraic of complexity $O_{M}(1)$. Note that \cite[Theorem 5.22]{BPR} does not explicitly bound the complexity of each Euclidean connected component, but the proof is algorithmic and yields a quantitative bound. Since the set $\R^n\backslash N$ from Lemma \ref{lem: effective inverse} is semi-algebraic of complexity $O_{M'}(1)$, and since each set $U_m$ from Lemma \ref{lem: effective inverse} is a union of $O_{M'}(1)$ connected components of $\R^n\backslash N$, we conclude that each set $U_m$ is semi-algebraic of complexity $O_{M'}(1)$.

It follows from the Tarski--Seidenberg theorem (see \cite[\S 2.2]{BCR}) that if $U\subseteq\mathbb{R}^n$ is semi-algebraic of complexity at most $M$, and if $\Phi=(\Phi_1,\ldots,\Phi_n)\colon\mathbb{R}^n\to\mathbb{R}^n$ with each $\Phi_i$ a polynomial of degree at most $M$, then both $\Phi(U)$ and $\Phi^{-1}(U)$ are semi-algebraic of complexity $O_M(1)$. In particular, since $\Phi^k$ is a tuple of polynomials, each of which have degree $O(1)$, we conclude that for $t,\bdr$ fixed, the sets $\Phi^k(U_m,t,\bdr)$ are semi-algebraic of complexity $O_{M'}(1)$ for $1 \leq m \leq M$.

Next, we recall \cite[Theorem 3$_n$, p.1047]{Pawlucki} (see also \cite[Corollary B]{Kurdyka}). The result below makes reference to a \textit{regular $L$-cell}. Informally, the definition is as follows. A one-dimensional regular $L$-cell is an interval. An $n$-dimensional regular $L$-cell is a set of points in $\R^n$ whose final coordinate is bounded above and below by the graphs of two functions whose partial derivatives have magnitude at most $L$; these two functions have a common domain, which is required to be a regular $L$-cell of dimension $n-1$.

\begin{lemma}\label{lem: L cell decomp}
Let $X \subset\mathbb{R}^n$ be open and semi-algebraic. Then we can decompose $X = N \sqcup \bigsqcup_{m=1}^M U_m$, where
\begin{enumerate}[i)]
    \item Each set $U_m$ is a regular $L$-cell, for some $L = O(1)$;
    \item Each set $U_m$ is semi-algebraic. Both $M$ and the complexity of each cell are bounded by a constant that depends only on $n$ and the complexity of $X$;
    \item $N$ is semi-algebraic and nowhere dense. In particular, $|N| = 0$.
\end{enumerate}
\end{lemma}
We remark that the results in \cite{Pawlucki} are stated in the slightly more general language of o-minimal structures, while we have stated the result in the special case where $X$ is a semi-algebraic set. Furthermore, Conclusion ii) above is not stated explicitly, but follows readily from the proof in \cite{Pawlucki}, which is explicit and algorithmic (the sets $U_m$ are constructed by cutting the set $X$ into pieces according to the vanishing locus of various ancillary polynomials obtained from the polynomials that define the set $X$; this procedure is explicitly described).

Finally, \cite[Proposition 1 (1), p. 1046]{Pawlucki} says that if $L>0$ and $U\subseteq \mathbb{R}^n$ is a regular $L$-cell, then for every pair of points $w,z\in U$, there is a continuous and piecewise $C^1$ curve $\gamma\colon[0, |w-z|]\to U$ with $\gamma(0)=w$, $\gamma(|w-z|) = z$, and $|\gamma'(s)|=O(L^{n-1})$ at each point for which $\gamma'(s)$ is defined.\footnote{Proposition 1 (1) from \cite{Pawlucki} says that the curve is continuous and definable. This in turn means that it is continuous and piecewise $C^1$, by \cite[Chapter 7, Proposition 2.5]{VdDries}.} This allows us to replace Conclusion i) from Lemma  \ref{lem: L cell decomp} with the following:
\begin{enumerate}
    \item[\emph{i$'$)}] For each set $U_m$ and each $w,z\in U_m$, there is a continuous, piecewise $C^1$ curve $\gamma\colon [0, |u-z|]\to U_m$ with $\gamma(0)=w$, $\gamma(|w-z|) = z$, and $|\gamma'(s)|=O(L^{n-1})$ at each point for which $\gamma'(s)$ is defined (in particular, $|\gamma'(s)|=O(L^{n-1})$ for all but finitely many values of $s$).
\end{enumerate}

\begin{proof}[Proof of Lemma~\ref{lem: comp diam}]  Fixing $(t, \bdr) \in (0,2] \times [1/2,2]^n$, we apply Lemma~\ref{lem: effective inverse} to the polynomial mapping $\Phi^k_{t, \bdr} \colon \R^n \to \R^n$ given by  $\Phi_{t, \bdr}^k(\omega) := \Phi^k(\omega, t, \bdr)$ for all $\omega \in \R^n$. This implies there exists some $M' \in \N$ with $M' = O(1)$ and open sets $U_1, \dots, U_{M'}$ such that
\begin{enumerate}[i)]
    \item The $\{U_m\}_{m=1}^{M'}$ are pairwise disjoint and cover $\R^n$ except for a null set;
    \item The restriction $\Phi_{t, \bdr}^k|_{U_m}$ of $\Phi_{t, \bdr}^k$ to $U_m$ is a diffeomorphism onto its image for $1 \leq m \leq M'$. 
\end{enumerate}

Now let $0 < \rho < 1$ and apply Lemma \ref{lem: L cell decomp} to each of the semi-algebraic sets $\Phi_{t, \bdr}^k( \Omega^k(t, \bdr; \rho) \cap U_m)$, $1 \leq m \leq M'$. Each of these sets has complexity $O(1)$. This yields a collection of semi-algebraic sets $\{V_{m,j}\}_{j, m = 1}^{M''}$ for some $M'' = O(1)$. After re-indexing (and redefining $M$), we denote the collection of semi-algebraic sets $\{\Phi_{t, \bdr}^k|_{U_m}^{-1}(V_{m,j})\}_{j, m = 1}^{M''}$ by $\{\Omega_m^k(t, \bdr; \rho)\}_{m=1}^M$. We have $M = O(1)$, and since each function $\Phi_{t, \bdr}^k|_{U_m}$ is a diffeomorphism, we have that \eqref{eq: comp diam} (that is, the assertion that $|N|=0$) follows from Conclusion i) of Lemma \ref{lem: effective inverse} and Conclusion iii) of Lemma \ref{lem: L cell decomp}. For each $1 \leq m \leq M$, let $\Phi_{m,t, \bdr}^k$ denote the restriction of $\Phi_{t, \bdr}^k$ to $\Omega_m^k(t, \bdr; \rho)$.

It remains to show that the sets $\Omega_m^k(t, \bdr; \rho)$ satisfy the diameter condition. Let $\Omega_m^k(t, \bdr; \rho)$ be non-empty, let $\omega$, $\xi \in \Omega_m^k(t, \bdr; \rho)$ be arbitrary points, and let $w := \Phi^k(\omega, t, \bdr)$, $z := \Phi^k(\xi, t, \bdr)$. Since $\Omega_m^k(t, \bdr; \rho) \subseteq \Omega^k(t,\bdr, \rho)$, it follows from the definitions that $|w| < \rho$ and $|z| < \rho$. By Lemma \ref{lem: L cell decomp} Conclusion i$'$), there is a continuous, piecewise $C^1$ curve $\gamma\colon[0, |w-z|]\to \Phi_{m,t, \bdr}^k(\Omega_m^k(t, \bdr; \rho))$ with $\gamma(0)=w,$ $\gamma(|w-z|) = z$, and $|\gamma'(s)| = O(1)$ for all but finitely many values of $s$. By the fundamental theorem of calculus, we have

\begin{equation*}
\begin{split}
    |\omega - \xi| & = \big|(\Phi_{m, t, \bdr}^k)^{-1}(w) - (\Phi_{m, t, \bdr}^k)^{-1}(z)\big| \\ 
    &= \Big|\int_0^{|w-z|} \big((\Phi_{m, t, \bdr}^k)^{-1}\big)'(\gamma(s))\gamma'(s)ds\Big|   \lesssim |w-z| \big\Vert |D(\Phi_{m, t, \bdr}^k)^{-1}|\big\Vert_{\infty}  
    \lesssim \rho/t,
    \end{split}
\end{equation*}
where the last inequality is  Lemma~\ref{lem: non deg} ii). This shows $\diam \Omega_m^k(t, \bdr; \rho) \lesssim \rho/t$, as required.
\end{proof}

We are now finally in position to conclude the proof of Lemma~\ref{lem: near root}. 

\begin{proof}[Proof of Lemma~\ref{lem: near root}] Fix $t \in (0,2]$ and $\bdr \in [1/2, 2]^n$. Since $E_{x, \bdr, < \rho}^{\delta, k}(0, \bdone)$ is contained in $B(0,2)$, we may assume without loss of generality that $\rho < c t$, where $c > 0$ is the dimensional constant appearing in the statement of Lemma~\ref{lem: non deg}.

We claim that 
\begin{equation}\label{eq: near root 1}
    E_{x, \bdr, < \rho}^{\delta, k}(0, \bdone) \subseteq \Omega^k(t, \bdr; \rho).
\end{equation}
Indeed, if $\omega \in E_{x, \bdr, < \rho}^{\delta, k}(0, \bdone)$, then it follows directly from the definition that $\omega \in \Omega^k$. On the other hand, the conditions $\|J_{x, \bdr}(\omega)\| < \rho$ and $\omega \in E^{\delta}(0,1)$, together with \eqref{eq: Cauchy--Binet}, imply that 
\begin{equation*}
\Big(\sum_{\substack{1 \leq j \leq n \\ j \neq k}}|G_j^k(\omega,t,\bdr)|^2\Big)^{1/2} < \rho/4 \qquad \textrm{and} \qquad |F_{0, \bdone}(\omega)| < \delta \leq \rho, 
\end{equation*}
 respectively. Thus, we have $|\Phi^k(\omega, t, \bdr)| < \rho$. Combining these observations with the definitions from \eqref{eq: X^k} and \eqref{eq: good omega set 2}, we obtain \eqref{eq: near root 1}. 

For $1 \leq m \leq M$, let $\Omega_m^k(t, \bdr; \rho)$ be the disjoint measurable sets featured in the decomposition of $\Omega^k(t, \bdr; \rho)$ from Lemma~\ref{lem: comp diam}. Furthermore, let $\Xi^k(t, \bdr; \rho) = \{\xi_1, \dots, \xi_M\}$ be an arbitrary set of representative elements for this decomposition, so that $\xi_m \in \Omega_m^k(t, \bdr; \rho)$ for $1 \leq m \leq M$. From Lemma~\ref{lem: comp diam}, we know that $\#\Xi^k(t, \bdr; \rho) = M \lesssim 1$ and there exists some dimensional constant $C_n \geq 1$ such that $\Omega_m^k(t, \bdr; \rho) \subseteq B(\xi_m, C_n\rho/t)$ for $1 \leq m \leq M$. Combining these observations with \eqref{eq: near root 1}, we conclude that
\begin{equation*}
    E_{x, \bdr, < \rho}^{\delta, k}(0, \bdone) \subseteq  \bigcup_{\xi \in \Xi^k(t, \bdr; \rho)} B(\xi, C_n\rho/t) \cup N 
\end{equation*}
for some null set $N \subseteq \R^n$, as required.     
\end{proof}




\section{Reduction to the discretised estimate}\label{sec: Fourier analysis} We conclude by showing how Theorem~\ref{thm: discretised} can be used to prove Theorem~\ref{thm: main}. There are two issues which need to be dealt with: the $\delta^{-\varepsilon}$-loss in the operator norm in \eqref{eq: discretised global} and the additional constraint that the principal radii lie in $[1,2]^n$.




\subsection{The local maximal operator}\label{subsec: loc max op} We first prove $L^p$ bounds for the local variant of $\cM_{\mathrm{st}}$ defined by
\begin{equation*}
    \cM_{\mathrm{loc}}f(x) := \sup_{\bdr \in [1,2]^n} |f \ast \sigma_{\bdr}(x)| \qquad \textrm{for $f \in C_c(\R^n)$.}
\end{equation*}
 For this, we reformulate Theorem~\ref{thm: discretised} as an $L^p$--Sobolev estimate and then interpolate with similar estimates from \cite{LLO}.\medskip
 
Let $\beta \in \cS(\R)$ be a Schwartz function with
\begin{equation*}
    \supp \beta \subseteq (1/2, 2) \qquad \textrm{such that} \qquad \sum_{j \in \Z} \beta(2^{-j}s) = 1 \qquad \textrm{for all $s \in \R \setminus \{0\}$. }
\end{equation*}For $j \in \Z$, let $\cP_j$ and $\cP_{< j}$ be the Fourier multiplier operators with symbols $\beta(2^{-j}\,\cdot\,)$ and $\sum_{i < j} \beta(2^{-i}\,\cdot\,)$, respectively.\medskip

By the $L^p$ boundedness of the (classical) strong maximal function,
\begin{equation*}
    \|\cM_{\mathrm{loc}} \circ \cP_{< 0}f\|_{L^p(\R^n)} \lesssim \|f\|_{L^p(\R^n)} \qquad \textrm{for all $1 < p < \infty$.}
\end{equation*}
For high frequencies, we have the following estimate. 

\begin{proposition}[$L^p$-Sobolev inequality]\label{prop: Lp Sobolev} For all $n \geq 3$ and $p > 2$, there exists an exponent $\eta(p) > 0$ such that
\begin{equation*}
    \|\cM_{\mathrm{loc}}(\cP_j f)\|_{L^p(\R^n)} \lesssim_p 2^{-j \eta(p)} \|f\|_{L^p(\R^n)}
\end{equation*}  
holds for all $f \in \cS(\R^n)$ and all $j \in \N_0$. 
\end{proposition}

\begin{proof} By \cite[Proposition 2.2]{LLO}, the proposition holds in the restricted range $p > 2 \cdot \frac{n+1}{n-1}$. Our goal is to show that, for all $\varepsilon > 0$, the inequality
\begin{equation}\label{eq: L2 Sobolev}
    \|\cM_{\mathrm{loc}}(\cP_j f)\|_{L^2(\R^n)} \lesssim_{\varepsilon} 2^{j \varepsilon} \|f\|_{L^2(\R^n)}
\end{equation}
holds for all $j \in \N$. Indeed, once \eqref{eq: L2 Sobolev} is established, it can be interpolated against the bounds from \cite[Proposition 2.2]{LLO} to obtain the desired estimate in the full range. However, \eqref{eq: L2 Sobolev} is in fact a standard (and fairly direct) consequence of Theorem~\ref{thm: discretised}, by dominating the kernel of $\cP_j$ in terms of normalised characteristic functions of balls: see, for instance, \cite[Lemma 5.1]{Schlag1997}.
\end{proof}




\subsection{From local to global} The observations of \S\ref{subsec: loc max op} show that the local maximal operator $\cM_{\mathrm{loc}}$ is bounded on $L^p$ for all $p > 2$. It remains to remove the restriction on the principal radii to obtain bounds for the global maximal operator $\cM_{\mathrm{st}}$. This is achieved using the multiparameter Littlewood--Paley argument from \cite[\S3]{LLO}. 

\begin{proof}[Proof of Theorem~\ref{thm: main}] The proof now follows precisely the same argument as that used in \cite[\S3]{LLO}. The only change is that at the top of page 10 of \cite{LLO} we appeal to Proposition~\ref{prop: Lp Sobolev} above rather than \cite[Proposition 2.2]{LLO}.     
\end{proof}

\appendix



\section{Necessary conditions}\label{app: necessity} Here we show that the condition $p > \frac{n+1}{n-1}$ is necessary for the $L^p$ boundedness of $\cM_{\mathrm{st}}$, justifying the numerology in Conjecture~\ref{conj: main}.

\begin{lemma}\label{lem: necessity} For $1 \leq p \leq \frac{n+1}{n-1}$, there exists some $f \in L^p(\R^n)$ and a measurable set $F \subseteq \R^n$ of positive $n$-dimensional Lebesgue measure such that 
\begin{equation}\label{eq: bad f}
    \cM_{\mathrm{st}}f(x) = \infty \qquad \textrm{for all $x \in F$.}
\end{equation}
\end{lemma}

Lemma~\ref{lem: necessity} relies on the following geometric observation, which makes precise the parameter counting heuristic from \S\ref{subsec: method}. Let $V$ denote the codimension $1$ linear subspace of $\R^n$ given by the orthogonal complement of the vector $N := \frac{1}{\sqrt{n}}(1, \dots, 1) \in \R^n$.

\begin{lemma}\label{lem: tangencies} There exists a measurable set $F \subseteq \R^n$ with positive $n$-dimensional measure such that the following holds. For all $x \in F$ there exists some $\bdr_x \in [1,2]^n$ such that:
 \begin{enumerate}[i)]
     \item $0 \in E(x, \bdr_x)$;
     \item The tangent plane to $E(x, \bdr_x)$ at $0$ agrees with $V$. 
 \end{enumerate}
\end{lemma}
 
Temporarily assuming Lemma~\ref{lem: tangencies}, we may establish Lemma~\ref{lem: necessity}.

\begin{proof}[Proof of Lemma~\ref{lem: necessity}] Let $g \colon \R^n \to \R$ be given by
\begin{equation*}
    g(x',x_n) := 
    \begin{cases}
        |x'|^{-(n-1)} |\log_2(1/|x'|)|^{-n/(n+1)} & \textrm{if $|x_n| \leq C|x'|^2$ and $|x'| \leq 1/2$,} \\
        0 & \textrm{otherwise}
    \end{cases}
\end{equation*}
for $x = (x', x_n) \in \R^{n-1} \times \R$. Here $C \geq 1$ is a dimensional constant, chosen sufficiently large to satisfy the forthcoming requirements of the proof. Now fix an orthonormal basis $u_1, \dots, u_n$ of $\R^n$ with $u_n = N$ and let $U$ denote the $n \times n$ orthogonal matrix with $j$th row equal to $u_j$ for $1 \leq j \leq n$. Finally, define
\begin{equation*}
  f := g \circ U.   
\end{equation*}
Since $\frac{n}{n+1} \cdot \frac{n+1}{n-1} = \frac{n}{n-1} > 1$, it is straightforward to check that $f \in L^p(\R^n)$ for all $1 \leq p \leq \frac{n+1}{n-1}$. It remains to show that this choice of $f$ satisfies \eqref{eq: bad f}.\medskip

Let $F \subseteq \R^n$ be the measurable set guaranteed by Lemma~\ref{lem: tangencies}. Fix $x \in F$, so that there exists some $\bdr \in [1,2]^n$ such that $E(x, \bdr)$ passes through $0$ and is tangent to $V$ at $0$. By the definition of the strong spherical maximal function, 
\begin{equation*}
    \cM_{\mathrm{st}}f(x) \geq \int_{E(x, \bdr)} |f(\omega)|\,\ud \sigma_{x, \bdr}(\omega) \geq \sum_{\ell = 1}^{\infty} \int_{E_{\ell}(x, \bdr)} |f(\omega)|\,\ud \sigma_{x, \bdr}(\omega)
\end{equation*}
where $\sigma_{x, \bdr}$ denotes the normalised surface measure on $E(x, \bdr)$ and
\begin{equation*}
    E_{\ell}(x, \bdr) := \big\{\omega \in E(x, \bdr) : 2^{-(\ell+1)/2} < \big|\mathrm{proj}_V\,\omega\big| \leq 2^{-\ell/2} \big\}.
\end{equation*}
The set $E_{\ell}(x, \bdr)$ consists of all points on the ellipsoid $E(x,\bdr)$ which lie at distance roughly $2^{-\ell/2}$ from $0$ in the tangential directions. In view of the tangency and the bounded curvature of the ellipsoid, these points are distance $O(2^{-\ell})$ from $0$ in the normal direction: in other words, $|\inn{\omega}{N}| \lesssim 2^{-\ell}$ for all $\omega \in E_{\ell}(x, \bdr)$. Thus, provided the constant $C \geq 1$ is chosen sufficiently large, $|\inn{\omega}{N}| \leq C|\mathrm{proj}_V\,\omega|$ for all $\omega \in E_{\ell}(x, \bdr)$ and so
\begin{equation*}
   |f(\omega)| = 
        |\mathrm{proj}_V\,\omega|^{-(n-1)/2} |\log_2(1/|\mathrm{proj}_V\,\omega|)|^{-n/(n+1)} \gtrsim 2^{\ell(n-1)/2} \ell^{-n/(n+1)}
\end{equation*}
for all $\omega \in E_{\ell}(x, \bdr)$. Since $\sigma(C_{x, \ell}) \gtrsim 2^{-\ell(n-1)/2}$ and $\frac{n}{n+1} < 1$, we deduce that
\begin{equation*}
      \cM_{\mathrm{st}}f(x) \gtrsim \sum_{\ell = 1}^{\infty} \ell^{-n/(n+1)} = \infty.
\end{equation*}
Thus $\cM_{\mathrm{st}}f(x) = \infty$ for all $x \in F$ and, as $F$ has positive measure, this concludes the proof.\end{proof}

It remains to verify the tangency lemma.

\begin{proof}[Proof of Lemma~\ref{lem: tangencies}] Given $\bdr \in [1,2]^n$, let $D_{\bdr}$ be the diagonal matrix such that the eigenvalue of the coordinate vector $e_j$ is $r_j^{-2}$. Then $F_{x,\bdr}(y) = \inn{D_{\bdr}(y-x)}{y - x}$ for all $y \in \R^n$, where $F_{x, \bdr}$ is the defining function from \eqref{eq: defining fn}. Conditions i) and ii) of the present lemma can be succinctly expressed as 
\begin{equation*}
     \inn{D_{\bdr}x}{x} = 1 \qquad \textrm{and} \qquad D_{\bdr} x = \lambda N \qquad \textrm{for some $\lambda \in \R \setminus \{0\}$.}
\end{equation*}
From this, it is not difficult to solve for $x$ in terms of $\bdr$ to deduce that
\begin{equation*}
    x_j = \frac{r_j^2}{|\bdr|} \quad \textrm{for $1 \leq j \leq n$} \quad \textrm{where} \quad |\bdr| := \Big(\sum_{j=1}^n r_j^2 \Big)^{1/2}
\end{equation*}
is the usual Euclidean norm.\medskip

Define the mapping $\Phi \colon \R^n \setminus \{0\} \to \R^n$ by 
\begin{equation*}
    \Phi(\bdr) := \frac{1}{|\bdr|}\big(r_1^2, \dots, r_n^2 \big) \qquad \textrm{for $\bdr \in \R^n \setminus \{0\}$.}
\end{equation*}
The Jacobian of this mapping is given by
\begin{equation*}
    J\Phi(\bdr) = \frac{1}{|\bdr|^3}
    \begin{bmatrix}
        2\sum_{j=1}^n r_j^3 - r_1^3 & - r_1^2 r_2 & \cdots & -r_1^2 r_n \\
      - r_1 r_2^2  & 2\sum_{j=1}^n r_j^3 - r_2^3 & \cdots & -r_2^2 r_n \\
        \vdots & \vdots & \ddots & \vdots \\
        - r_1 r_n^2  & -r_2 r_n^2 & \cdots &  2\sum_{j=1}^n r_j^3 - r_n^3 
    \end{bmatrix}.
\end{equation*}
We consider the special case $r_1 = \cdots = r_n = r$ and take the determinant to obtain
\begin{equation*}
    \det J\Phi(r, \dots, r) =  \frac{(-1)^n r^{3(n-1)}}{n^{3/2}} P(-(2n-1))
\end{equation*}
where
\begin{equation*}
P(a) :=
    \det \begin{bmatrix}
        a & 1 & \cdots & 1 \\
      1  & a & \cdots & 1 \\
        \vdots & \vdots & \ddots & \vdots \\
      1  & 1 & \cdots & a
    \end{bmatrix} \qquad  \textrm{for all $a \in \R$.}
\end{equation*}
It is a simple matter to show that $P(a) = (a-1)^{n-1}(a + n - 1)$ for all $a \in \R$. Indeed, it is clear that $P$ is a monic polynomial of degree $n$. By linear dependence relations between the columns, $P(1) = 0$ and, combining similar observations with the Leibniz rule, we further see that
\begin{equation*}
   P(1) = P^{(1)}(1) = \cdots = P^{(n-1)}(1) = 0.
\end{equation*}
Thus, $P(a) = (a - 1)^{n-1}(a - \xi)$ for some $\xi \in \R$. On the other hand, it is also clear that $P(-(n-1)) = 0$, since the sum of the columns of the matrix is the zero vector in this case. Hence we must have $\xi = -(n-1)$.\medskip

Using the above we see, for instance, that $|\det J\Phi(3/2, \dots, 3/2)| \gtrsim 1$ and so there exists some set open set $\Omega \subseteq [1,2]^n$ containing $(3/2, \dots, 3/2)$ upon which $\Phi$ is a bijection. Let $F := \Phi(\Omega)$ and for each $x \in F$ let $\bdr_x$ denote the unique $\bdr \in \Omega$ such that $\Phi(\bdr_x) = x$. By construction, it follows that the ellipsoid $E(x, \bdr_x)$ satisfies properties i) and ii) and, furthermore, $|F| \gtrsim 1$. 
\end{proof}




\bibliography{Reference}
\bibliographystyle{amsplain}

\end{document}